\documentclass[11pt,letterpaper]{article} 
\usepackage{amsthm}
\input{common.sty}
\input{froude.sty}

\title{The Froude number for solitary water waves with vorticity}
\begin{document}
\maketitle
\begin{abstract} 
  We consider two-dimensional solitary water waves on a shear flow
  with an arbitrary distribution of vorticity. Assuming that the
  horizontal velocity in the fluid never exceeds the wave speed and
  that the free surface lies everywhere above its asymptotic level, we
  give a very simple proof that a suitably defined Froude number $F$
  must be strictly greater than the critical value $F=1$. We also
  prove a related upper bound on $F$, and hence on the amplitude,
  under more restrictive assumptions on the vorticity.
\end{abstract}

\nowtheabstractisdone






\section{Introduction}\label{sec:intro} 
\subsection{Statement of the main results}\label{sec:intro:results} 

We consider the motion of a two-dimensional fluid which is bounded
above by a free surface under constant (atmospheric) pressure and
below by a horizontal bed. Gravity acts as an external force, and
there is no surface tension on the free surface. Inside the fluid, the
velocity $(u,v)$ and pressure $P$ satisfy the incompressible Euler
equations. We denote the horizontal bed by $y=-d$ and the free surface
by $y=\eta(x,t)$. Fixing the constant wave speed $c > 0$, we assume
that the motion is steady in that $\eta$, $u$, $v$, and $P$ depend on
$x$ and $t$ only through the combination $x-ct$, which we henceforth
abbreviate to $x$. We also assume that the wave is solitary in that
\begin{align*}
  \eta \to 0,
  \quad 
  v \to 0,
  \quad 
  u \to U(y),
  \qquad \asa x \to \pm\infty,
\end{align*}
uniformly in $y$, where the horizontal velocity $U$ of the shear flow
at $x=\pm\infty$ is an arbitrary function of $-d \le y \le 0$. We call
a solitary wave \emph{trivial} if $\eta \equiv 0$, $v \equiv 0$, and
$u \equiv U$. In the context of this paper, we will also call a
solitary wave a \emph{wave of elevation} if $\eta(x) \ge 0$
for all $x$ but $\eta \not \equiv 0$. Similarly we call a solitary
wave a \emph{wave of depression} if $\eta(x) \le 0$ for all
$x$ but $\eta \not \equiv 0$.

The classical Froude number for solitary waves is the dimensionless
ratio $c/\sqrt{gd}$. When working with shear flows, however, we find
it more convenient to define the Froude number $F$ by
\begin{align}
  \label{eqn:froude}
  \frac 1{F^2} = 
  g \int_{-d}^0 \frac {\dd y}{(c-U(y))^2},
\end{align}
where we have assumed that the velocity $U$ of the shear flow is
strictly less than the wave speed~$c$. This definition reduces to the
classical one when $U$ vanishes identically, and it has the advantage
that the critical Froude number is $F=1$ regardless of the shear flow
$U$. In particular, with this convention the small-amplitude solitary
waves with vorticity constructed in \citep{ter:rot,hur:exact,gw} have
Froude numbers $F$ slightly bigger than $1$. 

The reader may assume that $u,v,P,\eta,U$ are all $C^2$ or even
smooth. On the other hand, our arguments go through unchanged for
solutions with the more limited regularity
\begin{align}
  \label{eqn:reg}
  u,v,P \in W^{1,r}_\loc(\Dbar) \subset C^\alpha_\loc(\Dbar),
  \quad 
  \eta \in C^{1+\alpha}(\R),
  \quad 
  U \in W^{1,r}(-d,0) \sub C^\alpha[-d,0],
\end{align}
where here $0 < \alpha < 1$, $r = 2/(1-\alpha)$, and $D_\eta = \{(x,y)
: -d < y < \eta(x)\}$ denotes the fluid domain. By $w \in
W^{1,r}_\loc(\Dbar)$ we mean that $w \in W^{1,r}(D')$ whenever
$\overline{D'} \sub \Dbar$ is compact, and similarly for
$C^\alpha_\loc(\Dbar)$. The regularity \eqref{eqn:reg} is an analogue
for solitary waves of the regularity assumed in Theorem 2 of
\nat{\citealt}{\cite}{cs:discont} for periodic waves.
\begin{theorem}\label{thm:lower}
  Consider a solitary wave with $\sup u < c$ and the regularity
  \eqref{eqn:reg}. Then $F \ne 1$. Moreover, $F > 1$ if it is a wave
  of elevation, and $F < 1$ if it is a wave of depression.
\end{theorem}

In the irrotational case where the vorticity $\omega = v_x - u_y$
vanishes identically and $U$ is constant, the assumption $\sup u < c$
is automatically satisfied \citep{toland:survey} and the bound $F > 1$
for waves of elevation is well-known \citep{starr,at:finite,mcleod}.
While the assumption $\sup u < c$ is still reasonable for waves with
vorticity \citep{cs:exact}, it rules out the existence of critical
layers or stagnation points in flow. 

In some cases, the argument leading to Theorem~\ref{thm:lower} can be
extended to give an upper bound on the Froude number for waves of
elevation. Before giving this result, we define a dimensionless
quantity $\Lambda \ge 1$ by
\begin{align}
  \label{eqn:Lambda}
  \Lambda 
  = \max_y \frac{c-U(0)}{c-U(y)}.
\end{align}
We emphasize that $\Lambda$, like $F$, only depends on the shear flow
$U$ at infinity.
\begin{theorem}\label{thm:upper}
  For any solitary wave of elevation with the regularity
  \eqref{eqn:reg}, $\sup u < c$, and $\Lambda < 2/\sqrt 3$, the Froude
  number $F$ satisfies the upper bound
  \begin{align}
    \label{eqn:upper}
    F < \big( 1 - \tfrac 34 \Lambda^2 \big)^{-1/2}.
  \end{align}
\end{theorem}
For irrotational waves, $U$ is constant, so clearly $\Lambda = 1 <
2/\sqrt 3$ and hence Theorem~\ref{thm:upper} gives the well-known
bound $F < 2$ \citep{starr,at:finite,mcleod}. More generally, if the
vorticity $\omega \le 0$, then $U(y) \le U(0)$ for $-d \le y \le 0$ so
that again $\Lambda = 1$, and Theorem~\ref{thm:upper} gives the same
bound $F < 2$.
In terms of the antiderivative $\Gamma(p)$ of the vorticity
function and Bernoulli constant $\lambda$ defined in
Section~\ref{sec:prelim}, the condition $\Lambda < 2/\sqrt 3$ can be
rephrased as $\min_p\Gamma(p) > -\lambda/8$.

\subsection{Historical discussion} \label{sec:intro:hist} 


\paragraph{Irrotational waves.} 
For irrotational waves, the asymptotic shear flow $U$ is constant, and
can be taken to be zero by switching to an appropriate reference
frame. Our formula \eqref{eqn:froude} for $F$ then reduces to the
classical ratio
\begin{align*}
  F = \frac c{\sqrt{gd}},
\end{align*}
which is named in honor of William Froude, who in the 1870s argued
that in order to compare the resistances felt by scaled models of a
ship, the ratio of the speed of the ship to the square root of its
length must be kept constant \citep{froude:greyhound}. 

The importance of the critical speed $c=\sqrt{gd}$ corresponding to $F
= 1$ was known long before Froude. In 1781, Lagrange showed that long
irrotational waves in shallow water travel with nearly this speed
\citep{darrigol:horse}, and in 1828 Bélanger showed that a hydraulic
jump can occur only if the upstream flow has $F>1$
\citep{belanger:essai,chanson:belanger}. More pertinent to this article
is John Scott Russell's famous 1844 report \citep{russell}, which gives
the empirical formula
\begin{align}
  \label{eqn:russell}
  F^2 \approx 1 + \frac{\max \eta}d
\end{align}
for the speed of small-amplitude irrotational solitary waves.
Theoretical justifications of \eqref{eqn:russell} came decades later
with the work of Boussinesq in 1871 and Rayleigh in 1876
\citep{darrigol:horse}.

In 1947, Starr gave a formal proof\nat{}{ \citep{starr}} of the
strikingly simple exact formula
\begin{align}
  \label{eqn:toogood}
  F^2 = 1 + \frac 3{2d} \frac{\int \eta^2 \, \dd x}{\int \eta\, \dd x}
\end{align}
for irrotational solitary waves\nat{ \citep{starr}}{};
see \citew{lh:solitary} for an alternate derivation. Given this
identity, the upper and lower bounds $1<F<2$ for waves of elevation
are straightforward. Indeed, since $\eta \ge 0$ does not vanish
identically, \eqref{eqn:toogood} immediately implies $F > 1$. On the
other hand, \eqref{eqn:toogood} also implies the upper bound
\begin{align}
  \label{eqn:alphatop}
  F^2 < 1 + \frac 32 \frac{\max \eta}d.
\end{align}
Since $\max \eta \le F^2d/2$ by Bernoulli's law, \eqref{eqn:alphatop}
in turn implies $F^2 < 1 + 3F^2/4$ and hence $F < 2$. Substituting $F
< 2$ back into Bernoulli's law we also obtain the bound $\max \eta <
2d$ on the amplitude.

In fact, Starr showed the improved upper bound 
\begin{align}
  \label{eqn:alphabetter}
  F^2 < 1 + \frac{\max \eta}d.
\end{align}
in which the coefficient $3/2$ in \eqref{eqn:alphatop} has been
reduced to the $1$ appearing in the asymptotic formula
\eqref{eqn:russell}. See \citew{kp} for an alternate derivation.
Arguing as in the previous paragraph, \eqref{eqn:alphabetter} leads to
the bounds $F < \sqrt 2$ and $\max \eta < d$ on the Froude number and
amplitude. We note that the proofs of \eqref{eqn:alphabetter} in
\nat{\citet{starr} and \citet{kp}}{\cite{starr,kp}} do
not depend on the identity \eqref{eqn:toogood}.

\nat
{ \citet{at:finite} gave rigorous proofs of the bounds $1<F<2$ in
their construction of large-amplitude irrotational solitary waves,
which involves reformulating the water wave problem as a Nekrasov-type
integral equation on the free surface. They objected to the assumption
that the ``mass'' $\int\eta\, \dd x$ was finite in the earlier
proofs\nat{}{ \cite{starr,lh:solitary,kp}}, and instead, as
\citet{mcleod} puts it, ``take sixteen pages and much complicated
estimating of integrals to prove $F>1$ without the assumption of
finite mass''. In response, \cite{mcleod} showed that the earlier
proofs could be easily modified to avoid the assumption of finite
mass. This modified proof also shows that the mass is necessarily
finite. } 
{ Amick and Toland gave
rigorous proofs of the bounds $1<F<2$ in their construction of
large-amplitude irrotational solitary waves \cite{at:finite}, in which
the water wave problem is reformulated as a Nekrasov-type integral
equation on the free surface. They objected to the assumption that the
``mass'' $\int\eta\, \dd x$ was finite in the earlier proofs
\cite{starr,lh:solitary,kp}, and instead, as McLeod puts it
\cite{mcleod}, ``take sixteen pages and much complicated estimating of
integrals to prove $F>1$ without the assumption of finite mass''. In
response to \cite{at:finite}, McLeod \cite{mcleod} showed that the
earlier proofs could be easily modified to avoid the assumption of
finite mass. This modified proof also shows that the mass is
necessarily finite. }

For later reference we also mention a third upper bound 
\begin{align}
  \label{eqn:sosad}
  F^2 < \frac{2(1+\max \eta/d)^2}{2+\max \eta/d},
\end{align}
which was obtained by \nat{\citet{kp}}{Keady and Pritchard \cite{kp}} using maximum
principle arguments. It is easy to check that \eqref{eqn:sosad} is
strictly weaker than \eqref{eqn:toogood} and \eqref{eqn:alphabetter}.
In particular, it cannot be combined with Bernoulli's law to obtain a
bound on the Froude number which is independent of the amplitude $\max
\eta$.

Numerics suggest that there is a one-parameter family of irrotational
solitary waves connecting small-amplitude waves with $F$ slightly
bigger than $1$ and the so-called wave of greatest height, which has a
stagnation point at its crest where there is a corner with a \ang{120}
interior angle. The maximum value of the Froude number as well as the
maxima of mass, momentum, and energy for this family are all achieved
before the wave of greatest height is reached
\citep{miles:survey,lhf:mme}. The wave with maximum Froude number has
$F =
1.294$ and $\max\eta/d =
0.790$. There seems to be some disagreement about the precise value of
the maximum amplitude, but nevertheless a consensus that it is
approximately $\max\eta/d = 0.83$ and hence that the corresponding
Froude number is $F = 1.29$ \citep{miles:survey,lhf:mme,hvb:steep}. We
note that the existence of a wave of extreme form was proved
rigorously in \nat{\citet{at:finite} and
\citet*{aft}}{\cite{at:finite,aft}}, while the existence of bifurcation
or turning points in the connected set of solutions containing
small-amplitude solitary waves was proved in \citew{plotnikov:turning}.

\paragraph{Waves with vorticity.} 
With vorticity, the importance of the critical value $F=1$ for long
waves in shallow water was recognized by \nat{\citet{burns}}{Burns in
1953 \cite{burns}}; $F=1$ is sometimes called the ``Burns condition''.
The small-amplitude solitary waves constructed by \nat{\citet{ter:rot}
and later by \citet{hur:exact} and then \citet{gw}}{ Ter-Krikorov
\cite{ter:rot} in 1961 and later by Hur \cite{hur:exact} and then
Groves and Wahlén \cite{gw}} all have $F$ slightly bigger than $1$. 

Besides the existence results mentioned above, there are to our
knowledge no lower bounds in the literature on the Froude number for
solitary waves with vorticity. Moreover, the only upper bounds
\nat{(\citealt{kk:bounds}; \citealt{paper}; also see
\citealt{kn:comp})}{\cite{kk:bounds,paper} (also see \cite{kn:comp})}
are proved by maximum principle arguments and reduce to
\eqref{eqn:sosad} for irrotational waves. In particular, these bounds
seem not to lead to bounds on the Froude number which are independent
of the amplitude $\max \eta$.

There are, however, many results on solitary waves with vorticity
which require the assumption that $F > 1$. In particular, Hur proved
exponential asymptotics \citep{hur:symmetry} and analyticity of
streamlines \citep{hur:analyticity} for waves with $F > 1$, and
symmetry for waves with $F > 1$ which are also waves of elevation
\citep{hur:symmetry} (see \cite{matioc2:regsym} for related results
without the assumption $F > 1$).
Upper and lower bounds (or the lack thereof) on the Froude number also
feature prominently in the construction by the author of
large-amplitude solitary waves in \cite{paper}, where it was also
shown that waves with $F > 1$ are necessarily strict waves of
elevation in that $\eta(x) > 0$ for all~$x$. See Section~\ref{sec:imp}
for more on the consequences of our main results for the amplitude, elevation,
symmetry, monotonicity, and decay of solitary waves, and
Section~\ref{sec:large} for more on the existence of large-amplitude
waves.

With large positive constant vorticity, numerics seem to suggest the
existence of overhanging waves with arbitrarily large Froude number
\citep{vandenb:constant}, a phenomena which cannot occur for
irrotational waves. See Section~\ref{sec:const} for versions of
Theorems~\ref{thm:lower} and \ref{thm:upper} in the special case of
constant vorticity. Note that overhanging waves must have $u=c$
somewhere along their free surfaces, and hence do not satisfy the
hypothesis $\sup u < c$ of Theorems~\ref{thm:lower} and
\ref{thm:upper}. While the existence proofs in
\nat{\citet{ter:rot}, \citet{hur:exact}, and
\citet{gw}}{\cite{ter:rot,hur:exact,gw}} are restricted to waves with
$\sup u < c$, the qualitative results in \cite{kk:bounds} are not.

\subsection{Method of proof and further consequences} \label{sec:intro:further} 

\paragraph{Integral identities with vorticity.} 
%
The main ingredient in the proof of Theorem~\ref{thm:lower} is an
integral identity, Lemma~\ref{lem:lower}, which seems to be completely
new. In particular, we emphasize that this lemma is \emph{not} a
straightforward generalization of the identity \eqref{eqn:toogood}
used in the irrotational case. Indeed, following the proof of
\eqref{eqn:toogood} but retaining the extra terms coming from
vorticity, one obtains a different identity, Lemma~\ref{lem:classic}
(which we use in the proof of Theorem~\ref{thm:upper}). Under the
relatively strong assumption that vorticity is nonnegative,
Lemma~\ref{lem:classic} implies a lower bound on the relative speed
$c-U(0)$ of asymptotic shear flow at the free surface. Except in the
irrotational case where $U(y)$ is constant, however, this bound is not
sufficient to show that the Froude number $F$ is greater than the
critical value~$1$. Indeed, the definition \eqref{eqn:froude} of $F$
involves the values $U(y)$ for all $-d \le y \le 0$, not just $y=0$, and
the lower bound $F > 1$ is sharp for small-amplitude waves.

\paragraph{Existence of large-amplitude waves.} 
In \cite{paper}, the author constructed a connected set $\cm$ of
solitary waves of elevation whose asymptotic shear flows are given by
\begin{align}
  \label{eqn:sadfamily}
  U(y) = U(y;F) = c-FU\s(y)
\end{align}
for some arbitrary but fixed positive function $U\s$ satisfying a
normalization condition. It is easy to see from \eqref{eqn:sadfamily}
that the dimensionless parameter $\Lambda$ is constant along $\cm$;
when $\Lambda < 2/\sqrt 3$, we can combine Theorems~\ref{thm:lower}
and \ref{thm:upper} with the results in \cite{paper} to show that
there exists a sequence of waves in $\cm$ which approach stagnation in
that $\sup u_n \to c$. This significant improvement is presented in
detail in Section~\ref{sec:large}
(Corollary~\ref{cor:solitary:extra}), along with a related result for
$\Lambda \ge 2/\sqrt 3$ (Corollary~\ref{cor:solitary}).

\paragraph{Waves with surface tension.} 
For irrotational waves with surface tension,
\nat{\citet{ak:tension}}{Amick and\linebreak Kirchgässner \cite{ak:tension}}
prove the analogue
\begin{align}
  \label{eqn:toogood:ten1}
  F^2 = 1 + \frac 3{2d} \frac{\int \eta^2 \, \dd x}{\int \eta\, \dd x}
  + \frac\sigma{gd} 
  \frac{\int (\sqrt{1+\eta_x^2}-1)\, \dd x}{\int \eta\, \dd x}
\end{align}
of \eqref{eqn:toogood}. When the constant coefficient $\sigma$ of
surface tension is nonnegative, \eqref{eqn:toogood:ten1} implies that
solitary waves of elevation have $F > 1$ while solitary waves of
depression have $F < 1$. In Section~\ref{sec:tension}, we will show
that the same is true for waves with vorticity; indeed, with slightly
stronger regularity assumptions, Lemma~\ref{lem:lower} and
Theorem~\ref{thm:lower} continue to hold in the presence of surface
tension and with nearly identical proofs.

\subsection{Outline} \label{sec:intro:outline} 
In Section~\ref{sec:prelim}, we perform a standard change of variables
and collect some formulas which will be used later. The main benefit
of this change of variables is that it transforms the fluid domain
into a fixed infinite strip, enabling us to multiply the Euler
equations by functions defined in terms of the asymptotic shear flow
$U(y)$ appearing in the definition \eqref{eqn:froude} of the Froude
number $F$.

In Section~\ref{sec:lower}, we give a short and elementary proof of
Theorem~\ref{thm:lower}, based on an integral identity,
Lemma~\ref{lem:lower}, in the transformed variables. The argument is
perhaps even simpler than the argument in \cite{mcleod}, provided the
change of variables in Section~\ref{sec:prelim} is taken for granted.

In Section~\ref{sec:upper}, we prove the upper bound
Theorem~\ref{thm:upper} using two additional integral
identities. The first, Lemma~\ref{lem:classic}, is essentially
\eqref{eqn:toogood} with an extra term coming from the vorticity, and
is proved mostly in the original physical variables. The second,
Lemma~\ref{lem:extra} is proved similarly to Lemma~\ref{lem:lower},
but with a different test function. 

In Section~\ref{sec:imp}, we collect some implications of
Theorems~\ref{thm:lower} and \ref{thm:upper} for the amplitude, mass,
symmetry, monotonicity, and exponential decay of solitary waves of
elevation. 

In Section~\ref{sec:large}, we show how Theorems~\ref{thm:lower} and
\ref{thm:upper} can be used to improve the existence theory for
large-amplitude solitary waves with vorticity developed by the author
in \cite{paper}.

In Section~\ref{sec:const}, we specialize Theorems~\ref{thm:lower} and
\ref{thm:upper} to the case where the vorticity is constant and more
explicit formulas can be given.

Finally, in Section~\ref{sec:tension} we prove that, with slightly
modified regularity assumptions, Lemma~\ref{lem:lower} and
Theorem~\ref{thm:lower} still hold for waves with surface tension.

\section{Preliminaries}\label{sec:prelim} 

With the conventions from Section~\ref{sec:intro:results}, we assume
that $u,v,P \in W^{1,r}_\loc(\Dbar)$ satisfy the stationary
incompressible Euler equations 
\begin{subequations}\label{eqn:ww}
  \begin{align}
    \label{eqn:ww:x}
    (u-c)u_x + vu_y &= -P_x,\\
    \label{eqn:ww:y}
    (u-c)v_x + vv_y &= -P_y -g,\\
    \label{eqn:ww:div}
    u_x + v_y &= 0,
  \end{align}
  in $L^r_\loc(\Dbar)$, together with the
  boundary conditions 
  \begin{alignat}{2}
    \label{eqn:ww:bot}
    v&=0 &\quad& \ona y=-d,\\
    \label{eqn:ww:top}
    v&=(u-c)\eta_x && \ona y =\eta(x),\\
    \label{eqn:ww:dynamic}
    P &= 0 &&\ona y = \eta(x),
  \end{alignat}
  pointwise, and the asymptotic conditions
  \begin{align}
    \label{eqn:ww:asym}
    \eta \to 0,
    \quad 
    v \to 0,
    \quad 
    u \to U(y)
    \qquad \asa x \to \pm\infty,
  \end{align}
  uniformly in $y$. The free surface elevation $\eta$ has the
  regularity $\eta \in C^{1+\alpha}(\R)$.
\end{subequations}
Here for convenience we have normalized $P$ to vanish on the free
surface; $P$ therefore represents the difference between the pressure in
the fluid and the atmospheric pressure. 

Because of incompressibility \eqref{eqn:ww:div}, there exists a stream
function $\psi \in W^{2,r}_\loc(\Dbar) \sub C^{1+\alpha}_\loc(\Dbar)$
which satisfies $\psi_y = u-c$ and $\psi_x = v$. From now on we will
always assume $\sup u < c$, or equivalently $\sup \psi_y < 0$. 
By the kinematic boundary conditions
\eqref{eqn:ww:bot}--\eqref{eqn:ww:top}, $\psi$ is constant on $y=-d$
and $y=\eta(x)$. Thus the flux
\begin{align}
  \label{eqn:flux}
  m:= \psi(x,-d)-\psi(x,\eta(x)) = \int_{-d}^{\eta(x)} (c-u(x,y))\,\dd y
  = \int_{-d}^0 (c-U(y))\, \dd y
\end{align}
is independent of $x$. We normalize $\psi$ so that $\psi = 0$ on
$y=\eta(x)$ and $\psi = -m$ on $y=-d$. The vorticity $\omega$ is given
in terms of $\psi$ by
\begin{align*}
  \omega = v_x - u_y = -\Delta \psi = \gamma(\psi) 
\end{align*}
for some function $\gamma \in L^r[-m,0]$ called the \emph{vorticity
function} \citep{cs:discont}. 

Using 
\begin{align*}
  q = x,
  \qquad p = -\psi
\end{align*}
as independent variables, we can rewrite \eqref{eqn:ww} in terms of the
so-called height function $h(q,p)$ defined by
\begin{align}
  \label{eqn:socalled}
  h(x,-\psi(x,y)) = y+d.
\end{align}
The advantage of this formulation is that $h$ is defined on the fixed
domain
\begin{align*}
  \Omega := \{(q,p) : -m < p < 0 \}.
\end{align*}
Defining the asymptotic height function $H(p)$ in a similar way,
\begin{align}
  \label{eqn:H}
  H(-\Psi(y))=y+d,
  \qquad 
  \wherea \Psi(y) = \int_0^y (U(s)-c)\, \dd s, 
\end{align}
we have $H(0) = d$ and $H(-m) = 0$. Since \eqref{eqn:socalled} implies
$h_p\inv = -\psi_y$, we necessarily have $\inf h_p > 0$ and $\min H_p
> 0$.

The arguments in \cite{cs:discont} show that, under the crucial
assumption $\sup u < c$,  the solitary water wave problem
\eqref{eqn:ww} is equivalent to the system
\begin{subequations}\label{eqn:height}
  \begin{alignat}{2}
    \label{eqn:height:1}
    \bigg( - \frac{1+h_q^2}{2h_p^2} + \frac 1{2H_p^2} \bigg)_p
    + \bigg( \frac{h_q}{h_p} \bigg)_q
    &= 0 &\qquad& -m < p < 0,\\
    \label{eqn:height:2}
    \frac{1+h_q^2}{2h_p^2} - \frac 1{2H_p^2} + g(h-H) &= 0
    && \ona p=0,\\
    \label{eqn:height:3}
    h &= 0 && \ona p=-m,
  \end{alignat}
  for $h \in W^{2,r}_\loc(\Obar) \sub C^{1+\alpha}_\loc(\Obar)$ and
  $H \in W^{2,r}[-m,0]$ together with the asymptotic conditions
  \begin{align}
    \label{eqn:height:asym}
    h_p \to H_p,\  h_q \to 0 \asa q \to \pm\infty
    \qquad \text{uniformly in $p$}.
  \end{align}
\end{subequations}
The velocity field $(u,v)$ and free surface $\eta$ can be recovered
from $h$ via
\begin{align}
  \label{eqn:conversion}
  c-u = \frac 1{h_p}, 
  \qquad 
  v = - \frac{h_q}{h_p},
  \qquad 
  \eta(q) = h(q,0)-H(0). 
\end{align}
We note that the divergence-form equation \eqref{eqn:height:1}
expresses the balance of the $y$-component of momentum and that
\eqref{eqn:height:2} is Bernoulli's law evaluated restricted to the
free surface.

Defining the antiderivative $\Gamma \in W^{1,r}[-m,0] \sub
C^\alpha[-m,0]$ of the vorticity function $\gamma$ and the Bernoulli
constant $\lambda$ by
\begin{align}
  \label{eqn:lambda}
  \Gamma(p) = \zint p \gamma(-s)\, \dd s,
  \qquad 
  \lambda = (U(0)-c)^2 = \frac 1{H_p^2(0)},
\end{align}
we have the following useful relation between $\gamma$, $H$, and $U$
\begin{align}
  \label{eqn:useful}
  (U-c)^2(H(p)) = \frac 1{H_p^2(p)} = \lambda + 2\Gamma(p).
\end{align}
In particular, the Froude number $F$ is given in terms of $H$ by
\begin{align}
  \label{eqn:FH}
  \frac 1{F^2} = g\int_{-d}^0 \frac{\dd y}{(U(y)-c)^2}
  = g\int_{-m}^0 H_p^3(p)\, \dd p.
\end{align}
In addition, Bernoulli's law can be written as
\begin{align}
  \label{eqn:Bernoulli}
  P + \frac{(u-c)^2+v^2}2 + gy - \frac \lambda 2 - \Gamma(-\psi) \equiv 0.
\end{align}
Indeed, one can see that the left hand side is constant by
differentiating and using \eqref{eqn:ww:x}--\eqref{eqn:ww:y}. The fact
that this constant is zero follows by sending $x \to \pm\infty$ in the
the dynamic boundary condition \eqref{eqn:ww:dynamic}.

\section{Lower bound} \label{sec:lower} 

\begin{lemma} \label{lem:lower}
  Any solitary wave with $\sup u<c$ and the regularity \eqref{eqn:reg}
  satisfies
  \begin{align}
    \label{eqn:lower}
    \left(\frac 1{F^2}  - 1\right) \! \int_{-M}^M \eta \, \dd x
    + \int_{-M}^M \! \int_{-m}^0 \!
    \frac{H_p^3 h^2_q + (2h_p + H_p)(h_p-H_p)^2}{2h_p^2}\, \dd p\, \dd q
    \to  0
    \quad 
    \asa M \to \infty.
  \end{align}
  \begin{proof}
    Defining the function
    \begin{align}
      \label{eqn:Phi}
      \Phi(p) = \int_{-m}^p H_p^3(s)\,\dd s,
    \end{align}
    we note that \eqref{eqn:FH} implies $g\Phi(0) = 1/F^2$.
    Multiplying \eqref{eqn:height:1} by $\Phi$ and integrating by
    parts, we then have, for any $M > 0$,
    \begin{align}
      \notag
      0 &= 
      \int_{-M}^M \int_{-m}^0 \bigg[
      \Big( - \frac{1+h_q^2}{2h_p^2} + \frac 1{2H_p^2} \Big)_p \Phi
      + \Big( \frac{h_q}{h_p} \Big)_q \Phi \bigg]\, \dd p\, \dd q\\
      \label{eqn:dropme}
      &=
      \int_{-M}^M \int_{-m}^0 
      \Big( \frac{1+h_q^2}{2h_p^2} - \frac 1{2H_p^2} \Big) H_p^3
      \, \dd p\, \dd q
      +
      \frac 1{gF^2} \int_{-M}^M
      \Big( - \frac{1+h_q^2}{2h_p^2} + \frac 1{2H_p^2} \Big)(q,0) \,
      \dd q\\
      \notag
      &\qquad +
      \int_{-m}^0 \frac{h_q}{h_p} \Phi\, \dd p \bigg|^{x=M}_{x=-M}.
    \end{align}
    Since $h_q \to 0$ as $q \to \pm\infty$ by \eqref{eqn:height:asym},
    the third term in \eqref{eqn:dropme} vanishes as $M \to \infty$.
    Using the boundary condition \eqref{eqn:height:2} to simplify the
    second term, we obtain
    \begin{align}
      \label{eqn:premagic}
      \int_{-M}^M \int_{-m}^0 
      \Big( \frac{1+h_q^2}{2h_p^2} - \frac 1{2H_p^2} \Big) H_p^3
      \, \dd p\, \dd q
      +
      \frac 1{F^2} \int_{-M}^M (h(q,0)-H(0)) \,
      \dd q
      \to 0
    \end{align}
    as $M \to \infty$. Rewriting the first integrand in
    \eqref{eqn:premagic} as
    \begin{align}
      \label{eqn:magic}
      \Big( \frac{1+h_q^2}{2h_p^2} - \frac 1{2H_p^2} \Big) H_p^3
      &=
      -(h_p-H_p) + \frac{H_p^3 h^2_q + (2h_p + H_p)(h_p-H_p)^2}{2h_p^2},
    \end{align}
    we see that
    \begin{align}
      \label{eqn:postmagic}
      \begin{aligned}
        &\int_{-M}^M \int_{-m}^0 
        \Big( \frac{1+h_q^2}{2h_p^2} - \frac 1{2H_p^2} \Big) H_p^3
        \, \dd p\, \dd q\\
        &\qquad=
        -\int_{-M}^M (h(q,0)-H(0))\, \dd q 
        +\int_{-M}^M \int_{-m}^0 
        \frac{H_p^3 h^2_q + (2h_p + H_p)(h_p-H_p)^2}{2h_p^2}
        \, \dd p\, \dd q.
      \end{aligned}
    \end{align}
    Plugging \eqref{eqn:postmagic} into \eqref{eqn:premagic},
    rearranging terms, and using the identity $h(q,0)-H(0) = \eta(q)$
    from \eqref{eqn:conversion}, we obtain \eqref{eqn:lower} as
    desired.
  \end{proof}
\end{lemma}
\begin{proof}[Proof of Theorem~\ref{thm:lower}]\label{proof:lower}
  Consider a nontrivial solitary wave. Then $H_p$ and $h_p$ are strictly
  positive, and $h_q(q,0)=\eta_x(q)$ does not vanish identically. Thus the second
  integrand in \eqref{eqn:lower} is a nondecreasing function of $M$
  and is strictly positive for $M$ sufficiently large, and therefore the
  limit in \eqref{eqn:lower} implies
  \begin{align}
    \label{eqn:happy}
    \limsup_{M \to \infty}\left\{\left( \frac 1{F^2} - 1 \right)  \int_{-M}^M
    \eta \,\dd x\right\} < 0.
  \end{align}
  Since the left hand side of \eqref{eqn:happy} vanishes if $F=1$, we
  must have $F \ne 1$. For a wave of elevation, $\eta(x)
  \ge 0$ for all $x$ but $\eta \not \equiv 0$, so \eqref{eqn:happy}
  implies that the coefficient $1/F^2-1$ is strictly negative,
  i.e.~that $F > 1$. Similarly, for a wave of depression, $\eta(x)
  \le 0$ for all $x$ but $\eta \not \equiv 0$, so \eqref{eqn:happy}
  implies that $1/F^2 - 1$ is strictly positive, i.e.~that $F < 1$.
\end{proof}

\section{Upper bound} \label{sec:upper} 
In this section we will make use of the vorticity function
$\gamma(-p)$, Bernoulli constant $\lambda=(c-U(0))^2$, and
antiderivative $\Gamma(p)$ of $\gamma$ defined in
Section~\ref{sec:prelim}. We begin by giving a formula in our notation
for an invariant called the flow force.
\begin{lemma}\label{lem:force}
  For any solitary wave with $\sup u < c$ and the regularity
  \eqref{eqn:reg} and for any $x$, the flow force
  \begin{align}
    \label{eqn:S}
    S:= \int_{-d}^{\eta(x)} (P+(u-c)^2)(x,y)\, \dd y
    = -2 \int_{-m}^0 \gamma H\, \dd p + \lambda d + \frac{gd^2}2.
  \end{align}
  \begin{proof}
    That $S$ is independent of $x$ is well-known, and can be proved,
    for instance, by integrating the identity $(P+(u-c)^2)_x +
    ((u-c)v)_y = 0$ over a region of the form $\{(x,y) : a < x < b,\
    -d < y < \eta(x)\}$ using the divergence theorem and then applying
    the boundary conditions. To obtain the formula \eqref{eqn:S}, we
    use Bernoulli's law \eqref{eqn:Bernoulli} to rewrite
    \begin{align*}
      P + (u-c)^2 
      = \frac{(u-c)^2-v^2}2 - gy + \frac \lambda 2 + \Gamma
      = \frac{1-h_q^2}{2h_p^2} -g(h-d)+ \frac \lambda 2 + \Gamma.
    \end{align*}
    The asymptotic condition \eqref{eqn:height:asym} gives
    \begin{align*}
      \frac{1-h_q^2}{2h_p^2} \to \frac 1{2H_p^2} = \frac \lambda 2 + \Gamma 
      \asa q \to \pm\infty,
    \end{align*}
    uniformly in $p$, and hence, since $S$ is independent of $x$,
    \begin{align}
      \notag
      S &= 
      \lim_{q\to\pm\infty}  
      \int_{-m}^0 \Big(\frac{1-h_q^2}{2h_p^2}
      -g(h-d)
      + \frac \lambda 2 + \Gamma
      \Big)h_p\,\dd p\\
      \notag
      &=
      \int_{-m}^0 \left(\lambda + 2\Gamma - g(H-d) \right) H_p\, \dd p,\\
      &=
      -2 \int_{-m}^0 \gamma H\, \dd p + \lambda d + \frac{gd^2}2,
    \end{align}
    where in the last step we integrated by parts using
    $H(-m) = \Gamma(0) = 0$ and $\Gamma_p = \gamma$.
  \end{proof}
\end{lemma}
The following integral identity is \eqref{eqn:toogood} but with an
extra term coming from the vorticity. Unlike Lemmas~\ref{lem:lower}
and \ref{lem:extra}, it is proved mostly in the original physical
variables.
\begin{lemma} \label{lem:classic}
  Any solitary wave with $\sup u<c$ and the regularity \eqref{eqn:reg}
  satisfies
  \begin{align}
    \label{eqn:classic}
    (\lambda - gd)\int_{-M}^M \eta\, \dd x 
    -\frac{3g}2 \int_{-M}^M \eta^2 \, \dd x
    -2\int_{-M}^M \int_{-m}^0 \gamma (h-H) \,\dd p\,\dd q
    \to 0
    \quad 
    \asa M \to \infty.
  \end{align}
  \begin{proof}
    Consider the fluid region
    \begin{align*}
      D = \{(x,y) \in \R^2 : -M < x < M,\ -d < y < \eta(x)\},
    \end{align*}
    and the two $W^{1,r}(D)$ vector fields 
    \begin{align*}
      A &= (A^1,A^2) = \big(P+(u-c)^2,(u-c)v\big),\\
      B &= (B^1,B^2) = \big((u-c)v,P+v^2+gy\big).
    \end{align*}
    By the incompressible Euler equations
    \eqref{eqn:ww:x}--\eqref{eqn:ww:div}, $A$ and $B$ are both
    divergence free. Thus
    \begin{align}
      \notag
      \div (xA+(y+d)B) 
      &=
      A^1 + B^2 \\
      \notag
      &= 
      2P + (u-c)^2 + v^2 + gy \\
      &= 
      \label{eqn:abdiv}
      \lambda + 2\Gamma(-\psi) - gy
    \end{align}
    in $L^r(D)$, where in the last step we have used Bernoulli's law
    \eqref{eqn:Bernoulli}. Integrating \eqref{eqn:abdiv} over $D$, the
    divergence theorem gives
    \begin{align}
      \label{eqn:abgreen}
      \int_{\dell D} (xA+(y+d)B) \cdot n \, \dd s
      = \iint_D (\lambda - gy + 2\Gamma(-\psi)) \, \dd y\, \dd x,
    \end{align}
    where $n$ is an outward pointing normal. We will obtain
    \eqref{eqn:classic} by simplifying both sides of
    \eqref{eqn:abgreen} and using Lemma~\ref{lem:force}.

    First consider the left hand side of \eqref{eqn:abgreen}. On the
    free surface, the two boundary conditions $v=\eta_x (u-c)$ and
    $P=0$ give
    \begin{align*}
      (xA+(y+d)B) \cdot n
      = (xA + (y+d)B) \cdot \frac {(v,c-u)}{\sqrt{(u-c)^2+v^2}}
      = \frac {gy(y+d)(c-u)}{\sqrt{(u-c)^2+v^2}},
    \end{align*}
    while on the bottom $y=-d$ the boundary condition $v=0$ implies $A^2 = (y+d)B^2
    = 0$. Thus we see
    \begin{align}
      \label{eqn:abndry}
      \begin{aligned}
        \int_{\dell D} (xA+(y+d)B) \cdot n \, \dd s
        &= 
        x\int_{-d}^{\eta(x)}
        (P+(u-c)^2) \,
        \dd y 
        \bigg|^{x=M}_{x=-M}\\
        &\qquad +
        \int_{-d}^{\eta(x)}
         (y+d)(u-c)v 
         \, \dd y 
        \bigg|^{x=M}_{x=-M}
        +
        \int_{-M}^M g(\eta+d) \eta\, \dd x.
      \end{aligned}
    \end{align}
    By Lemma~\ref{lem:force}, we have 
    \begin{align*}
      \int_{-d}^{\eta(x)} \big(P+(u-c)^2\big) \, \dd y
      = -2\int_{-m}^0 \gamma(-p) H\, \dd p
      + \lambda d + \frac{gd^2}2
    \end{align*}
    for all $x$,
    and since $v \to 0$ uniformly in $y$ as $x \to \pm\infty$, the
    second term on the right hand side of \eqref{eqn:abndry} vanishes
    as $M \to \infty$. Thus \eqref{eqn:abndry} implies
    \begin{align}
      \label{eqn:abndrylim}
      \begin{aligned}
        &
        \int_{\dell D} (xA+(y+d)B) \cdot n \, \dd s
        - g \int_{-M}^M \eta^2 \,\dd x
        \\ &\qquad
        - gd \int_{-M}^M \eta \,\dd x
        - 2M\left(
        \lambda d + \frac{gd^2}2
        - 2 \int_{-m}^0 \gamma H\, \dd p
        \right)
        \to 0
      \end{aligned}
    \end{align}
    as $M \to \infty$.

    Now we turn to the right hand side of \eqref{eqn:abgreen}. 
    Changing variables and integrating by parts, 
    \begin{align*}
      \int_{-d}^{\eta(x)} \Gamma(-\psi)\, \dd y
      = \int_{-m}^0 \Gamma(p) h_p\, \dd p
      = -\int_{-m}^0 \gamma(-p) h\, \dd p,
    \end{align*}
    where we have used that $\Gamma$ vanishes on $p=0$ while $h$
    vanishes on $p=-m$. Thus
    \begin{align}
      \label{eqn:abulk}
      \begin{aligned}
        \iint_D (\lambda -gy + 2\Gamma(-\psi))\, \dd y\, \dd x 
        &= \lambda \int_{-M}^M \eta\, \dd x
        - \frac g2 \int_{-M}^M \eta^2\, \dd x 
        - 2 \int_{-M}^M \int_{-m}^0 \gamma h\, \dd p\, \dd q\\
        &\qquad + 2M\left(\lambda d + \frac{gd^2}2
        \right)
      \end{aligned}
    \end{align}
    Substituting \eqref{eqn:abndrylim} and \eqref{eqn:abulk} into
    \eqref{eqn:abgreen}, most of the terms drop out and we are left
    with \eqref{eqn:classic} as desired.
  \end{proof}
\end{lemma}

\begin{lemma} \label{lem:extra}
  Any solitary wave with $u<c$ and the regularity \eqref{eqn:reg} satisfies 
  \begin{align}
    \label{eqn:extra}
    \frac{3g}2 \int_{-M}^M \eta^2 \, \dd x
    - \int_{-M}^M \int_{-m}^0 
    \frac{H_p^3 h^2_q + (2h_p + H_p)(h_p-H_p)^2}{2H_p^2 h_p^2}\, \dd p\, \dd q
    \to  0
    \quad 
    \asa M \to \infty.
  \end{align}
  \begin{proof}
    We argue as in the proof of Lemma~\ref{lem:lower}, but with the
    function $\Phi$ replaced by $H$, and then appeal to
    Lemma~\ref{lem:classic}. Multiplying \eqref{eqn:height:1} by $H$
    and integrating by parts, we have, for any $M > 0$,
    \begin{align}
      \notag
      0 &= 
      \int_{-M}^M \int_{-m}^0 \bigg[
      \Big( - \frac{1+h_q^2}{2h_p^2} + \frac 1{2H_p^2} \Big)_p H
      + \Big( \frac{h_q}{h_p} \Big)_q H \bigg]\, \dd p\, \dd q\\
      \label{eqn:dropmeagain}
      &=
      \int_{-M}^M \int_{-m}^0 
      \Big( \frac{1+h_q^2}{2h_p^2} - \frac 1{2H_p^2} \Big) H_p
      \, \dd p\, \dd q
      +
      gd\int_{-M}^M
      \eta\,
      \dd q
      +
      \int_{-m}^0 \frac{h_q}{h_p} H\, \dd p \bigg|^{x=M}_{x=-M},
    \end{align}
    where we have used the boundary condition \eqref{eqn:height:2} as
    well as $h(q,0)-H(0)=\eta(q)$ and $H(0)=d$. As in the proof of
    Lemma~\ref{lem:lower}, the asymptotic conditions
    \eqref{eqn:height:asym} imply that the last term in
    \eqref{eqn:dropmeagain} vanishes as $M \to \infty$. The integrand
    in the first term of \eqref{eqn:dropmeagain} can be rewritten as
    \begin{align}
      \label{eqn:magicagain}
      \Big( \frac{1+h_q^2}{2h_p^2} - \frac 1{2H_p^2} \Big) H_p
      &=
      -\frac{h_p-H_p}{H_p^2} + \frac{H_p^3 h^2_q + (2h_p + H_p)(h_p-H_p)^2}{2h_p^2 H_p^2},
    \end{align}
    as can be seen by dividing \eqref{eqn:magic} by $H_p^2$.
    Since
    \begin{align*}
      \Big(\frac 1{H_p^2}\Big)_p = (\lambda + 2\Gamma)_p = 2\gamma,
      \qquad 
      \frac 1{H_p^2(0)} = \lambda,
    \end{align*}
    we can integrate the first term of \eqref{eqn:magicagain} by parts to get
    \begin{align*}
      - \int_{-m}^0 \frac{h_p-H_p}{H_p^2}(q,p) \, \dd p
      &= 2\int_{-m}^0 \Big(\frac 1{H_p^2}\Big)_p (h(q,p)-H(p)) \, \dd p - \lambda (h(q,0)-H(0))\\
      &= 2\int_{-m}^0 \gamma (h(q,p)-H(p)) \, \dd p - \lambda \eta(q).
    \end{align*}
    Putting everything together, we see that \eqref{eqn:dropmeagain} implies
    \begin{align}
      \label{eqn:extraterm}
      \begin{aligned}
        &\int_{-M}^M \int_{-m}^0 
        \frac{H_p^3 h^2_q + (2h_p + H_p)(h_p-H_p)^2}{2h_p^2 H_p^2}
        \, \dd p\, \dd q
        +2 \int_{-M}^M \int_{-m}^0 
        \gamma (h-H) 
        \, \dd p\, \dd q\\
        &\qquad 
        + (gd-\lambda)\int_{-M}^M \eta\, \dd q
        \longrightarrow 0
        \qquad \asa M \to \infty
      \end{aligned}
    \end{align}
    Adding \eqref{eqn:extraterm} with the conclusion
    \eqref{eqn:classic} of Lemma~\ref{lem:classic}, most of the terms
    cancel and we are left with \eqref{eqn:extra} as desired.
  \end{proof}
\end{lemma}

Using Lemmas~\ref{lem:lower} and \ref{lem:extra}, we can now prove
Theorem~\ref{thm:upper}.
\begin{proof}[Proof of Theorem~\ref{thm:upper}]
  The key observation is that the integrands of the double
  integrals in \eqref{eqn:lower} and \eqref{eqn:extra} differ only by
  a factor of $H_p^2$. 
  Plugging the crude estimate
  \begin{align*}
    &\int_{-M}^M \int_{-m}^0 
    \frac{H_p^3 h^2_q + (2h_p + H_p)(h_p-H_p)^2}{2h_p^2}
    \, \dd p\, \dd q\\
    &\qquad\qquad\le 
    \max H_p^2
    \int_{-M}^M \int_{-m}^0 
    \frac{H_p^3 h^2_q + (2h_p + H_p)(h_p-H_p)^2}{2h_p^2H_p^2}
    \, \dd p\, \dd q
  \end{align*}
  into the result \eqref{eqn:lower} of Lemma~\ref{lem:lower}, we
  obtain
  \begin{align*}
    \limsup_{M \to \infty} 
    &
    \Bigg\{
    \left(1- \frac 1{F^2}\right) \int_{-M}^M \eta \, \dd x
    \\ &\qquad
    - \max H_p^2 \int_{-M}^M \int_{-m}^0 
    \frac{H_p^3 h^2_q + (2h_p + H_p)(h_p-H_p)^2}{2h_p^2H_p^2}
    \, \dd p\, \dd q
    \Bigg\}
    \le 0.
  \end{align*}
  Subtracting the conclusion \eqref{eqn:extra} of Lemma~\ref{lem:extra}
  multiplied by $\max H_p^2$, we then get
  \begin{align}
    \label{eqn:liminf}
    \limsup_{M \to \infty} \left\{
    \bigg(1- \frac 1{F^2}\bigg) \int_{-M}^M \eta \, \dd x
    -\frac{3 g\max H_p^2}2 \int_{-M}^M 
    \eta^2 \, \dd q
    \right\}
    \le 0.
  \end{align}
  But from Bernoulli's law \eqref{eqn:Bernoulli} evaluated on the free
  surface we have $g\max \eta < \lambda/2$, so that \eqref{eqn:liminf} implies
  \begin{align*}
    \limsup_{M \to \infty}
    \bigg(1- \frac 1{F^2} -\frac{3\lambda} 4 \max H_p^2 \bigg) \int_{-M}^M \eta \, \dd x
    < 0,
  \end{align*}
  and hence
  \begin{align}
    \label{eqn:prearranged}
    1- \frac 1{F^2} -\frac{3 \lambda}4 \max H_p^2 < 0.
  \end{align}
  By \eqref{eqn:lambda} and \eqref{eqn:useful},
  $\lambda = (U-c)^2(0)$ and $H_p^2 = (U-c)^{-2}$,
  so 
  \begin{align*}
    \lambda \max_p H_p^2 = \max_y \frac{(c-U(0))^2}{(c-U(y))^2} =
    \Lambda^2,
  \end{align*}
  and \eqref{eqn:prearranged} becomes
  \begin{align*}
    1- \frac 1{F^2} - \frac{3\Lambda^2}4 < 0,
  \end{align*}
  which, assuming $\Lambda < 2/\sqrt{3}$, is equivalent to the desired
  upper bound on $F$ in \eqref{eqn:upper}.
\end{proof}

\section{Amplitude, elevation, symmetry, monotonicity, and decay} \label{sec:imp}


In this section we will give several corollaries of
Theorems~\ref{thm:lower}--\ref{thm:upper} and their proofs. Some of
these will require stronger regularity and decay assumptions than
\eqref{eqn:reg} and \eqref{eqn:ww:asym}, namely
\begin{gather}
  \label{eqn:regmore}
  h \in C^{2+\alpha}(\Obar),
  \qquad 
  H \in C^{2+\alpha}[-d,0],\\
  \label{eqn:asymmore}
  h-H,\ D(h-H),\ D^2(h-H) \to 0 \asa q \pm \infty,
\end{gather}
where the limit in \eqref{eqn:asymmore} is uniform in $p$. 


\begin{corollary}[Bound on the amplitude]\label{cor:height}
  In the setting of Theorem~\ref{thm:upper}, the maximum amplitude
  $\max_x \eta(x)$ satisfies the following upper bound
  \begin{align}
    \label{eqn:ugly}
    \frac{\max \eta}d < \frac{(c-U(0))^2}{2gd} < \frac 12
    \frac{\Lambda^2}{1- \frac 34 \Lambda^2}.
  \end{align}
  \begin{proof}
    The first inequality in \eqref{eqn:ugly} is just Bernoulli's law
    \eqref{eqn:Bernoulli} evaluated on the free surface. Next, we note
    that the definitions \eqref{eqn:froude}, \eqref{eqn:Lambda}, and
    \eqref{eqn:lambda} of $F$ and $\Lambda$ immediately imply the
    simple inequality $(c-U(0))^2 \le g d\Lambda^2 F^2$. The second
    inequality in \eqref{eqn:ugly} then follows from the upper bound
    on $F$ in Theorem~\ref{thm:upper}. 
  \end{proof}
\end{corollary}
When $U(y) \le U(0)$ for $-d \le y \le 0$, such as for instance when
$\omega \le 0$ so that $U_y \ge 0$, we have $\Lambda = 1$ and hence
that the second inequality in \eqref{eqn:ugly} is the simple bound
$\max \eta \le 2d$. See Section~\ref{sec:const} for the case of
constant vorticity.
\begin{corollary}[Elevation]\label{cor:elevation}
  A solitary wave with $\sup u < c$ and the regularity
  \eqref{eqn:regmore} and decay \eqref{eqn:asymmore} is a wave of
  elevation if and only if $F > 1$, and in this case it is a strict
  wave of elevation in that $\eta(x) > 0$ for all $x \in \R$.
  \begin{proof}
    Thanks to Theorem~\ref{thm:lower}, this is an immediate
    consequence of Proposition~2.1 in \cite{paper}, which states that
    all waves with $F \ge 1$ and the above regularity and decay are
    strict waves of elevation.
  \end{proof}
\end{corollary}

While there are no symmetric and monotone irrotational solitary waves
of depression \cite{kp}, it is an open question if the same is true
with vorticity. More generally, it is unknown if there are any
solitary waves with subcritical Froude number $F < 1$. By
Theorem~\ref{thm:lower}, no such wave could be a wave of elevation,
and in fact \eqref{eqn:lower} implies something stronger.
\begin{corollary}[Negative mass for subcritical waves]\label{cor:depression}
  Any nontrivial solitary wave with $\sup u < c$, the regularity
  \eqref{eqn:reg}, and subcritical Froude number $F < 1$ must have 
  \begin{align*}
    \limsup_{M\to \infty} \int_{-M}^M \eta\, \dd x < 0.
  \end{align*}
  \begin{proof}
    This follows immediately from \eqref{eqn:happy}.
  \end{proof}
\end{corollary}

The following corollary is
interesting only in the method of proof; the conclusion of
Corollary~\ref{cor:hur}, which follows from \cite{hur:symmetry}, is
much stronger.
\begin{corollary}[Finite mass]\label{prop:mass}
  For any solitary wave of elevation with $\sup u < c$ and the
  regularity \eqref{eqn:reg}, all of the definite integrals appearing
  in Lemmas~\ref{lem:lower}, \ref{lem:classic}, and \ref{lem:extra},
  have a (finite) limit as $M \to \infty$. In particular, the limits
  in \eqref{eqn:lower}, \eqref{eqn:classic}, and \eqref{eqn:extra}
  become equalities when $M$ is replaced by $+\infty$.
  \begin{proof}
    We argue as in Section~3 of \cite{mcleod}. Assume for
    contradiction that 
    \begin{align*}
      \int_{-M}^M \eta\, \dd x \to \infty 
      \asa M \to \infty.
    \end{align*}
    Since $\eta \to 0$ as $x \to \pm\infty$, we have 
    \begin{align*}
      \frac{\int_{-M}^M \eta^2\, \dd x }{\int_{-M}^M \eta\, \dd x} 
      \to 0 \asa M \to \infty.
    \end{align*}
    But then, since $F > 1$ by Theorem~\ref{thm:lower}, the left hand
    side of \eqref{eqn:liminf} tends to $+\infty$ as $M \to \infty$, a
    contradiction. Thus $\nint\infty \eta\, \dd x$ and hence $\nint\infty
    \eta^2 \, \dd x$ are both finite. The statement then follows by
    combining this result with the limits in \eqref{eqn:lower},
    \eqref{eqn:classic}, and \eqref{eqn:extra}.
  \end{proof}
\end{corollary}

In order to state the final corollary in this section, we introduce
the Sturm-Liouville problem
\begin{align}
  \label{eqn:sturm}
  \begin{aligned}
    &\left(\frac{\varphi_p}{H_p^3}\right)_p + \mu \frac{\varphi_p}{H_p}
    = 0 \fora {-m} < p < 0,\\
    &\frac{\varphi_p(0)}{H_p^3} - g\varphi(0) = 0,
    \qquad 
    \varphi(-m) = 0,
  \end{aligned}
\end{align}
which appears when studying the linearization of \eqref{eqn:height}
around $h=H$. We define $\mu_1,\mu_2$ to be the smallest and second
smallest eigenvalues of \eqref{eqn:sturm}, and $\varphi_1$ to be the
eigenfunction corresponding to $\mu_1$. When $F > 1$, $\mu_1$ is
positive, and we can take $\varphi_1$ to be positive for $-m < p \le
0$ \citep{hur:symmetry}. We note that \eqref{eqn:sturm} is equivalent
the system
\begin{align*}
  \begin{aligned}
    &(U-c)(\til\varphi_{yy} + \mu\til\varphi) - U_{yy}\til\varphi = 0
    \fora {-d} < y < 0,\\
    &(U-c)^2 \til\varphi_y(0) - (g+(U-c)U_y)\til\varphi(0) = 0,
    \qquad 
    \til\varphi(-d) = 0,
  \end{aligned}
\end{align*}
for $\til\varphi(y) := (c-U(y))\varphi(p)$ in the original physical
variables; see Lemma~2.3 in \cite{hl:running}.
\begin{corollary}[Symmetry, monotonicity, and decay]\label{cor:hur}
  Any solitary wave of elevation with $\sup u < c$ and the regularity
  \eqref{eqn:regmore} and decay \eqref{eqn:asymmore} has the following
  properties. 
  \nat{\begin{enumerate}}
    {\begin{enumerate}[label=\textup{(\alph*)}]}
  \item[\rm (a)] \label{cor:hur:sym} (Symmetry and monotonicity) 
    The wave is symmetric and monotone in that, after shifting the
    definition of the horizontal variable $q$, the height function $h$
    is even in $q$ and has $h_q < 0$ for $q > 0$ and $-m < p \le 0$.
    In particular, after this shift $\eta$ is an even function of $x$
    with $\eta_x < 0$ for $x > 0$.
  \item[\rm (b)] \label{cor:hur:decay} (Decay and asymptotics) 
    The difference $w:= h-H$ decays exponentially as $q\to\pm\infty$,
    and satisfies the asymptotic estimate
    \begin{align*}
      \abs{D^k(w(q,p)-r\varphi_1(p)e^{-\sqrt{\mu_1}\abs q})}
      \le 
      C e^{-s_1 \abs q}
      \qquad \fora k \le 1, \ \abs q > 1,
    \end{align*}
    for some constants $C,r > 0$ depending on $w$, where $\mu_1 > 0$
    and $\varphi_1$ are defined above in terms of the Sturm-Liouville
    problem \eqref{eqn:sturm} and the exponent $s_1$ appearing on the
    right hand side satisfies $\sqrt{\mu_1} < s_1 <
    \min(2\sqrt{\mu_1},\sqrt{\mu_2})$.
  \end{enumerate}
  \begin{proof}
    Since the waves considered in this corollary have $F > 1$ by
    Theorem~\ref{thm:lower}, part \nat{(a)}{\ref{cor:hur:sym}} follows
    immediately from Theorem~3.1 in \cite{hur:symmetry}, while part
    \nat{(b)}{\ref{cor:hur:decay}} follows from Proposition 4.6 in the
    same paper.
  \end{proof}
\end{corollary}

\section{Existence of large-amplitude waves} \label{sec:large} 

In this section we observe the implications of
Theorems~\ref{thm:lower} and \ref{thm:upper} for the existence theory
of large-amplitude solitary waves with vorticity developed in
\cite{paper}. This theory involves a one parameter family of shear
flows 
\begin{align}
  \label{eqn:family}
  U(y) =  c - FU\s(y),
\end{align}
where $U\s$ is an arbitrary but fixed strictly positive function satisfying the
normalization condition
\begin{align}
  \label{eqn:normalized}
  g \int_{-d}^0 \frac{\dd y}{U\s(y)^2} = 1.
\end{align}
The normalization \eqref{eqn:normalized} ensures that the parameter
$F$ in \eqref{eqn:family} is indeed the Froude number $F$ defined in
\eqref{eqn:froude}. In the corollary below we call a solitary wave
\emph{symmetric} if $u$ and $\eta$ are even in $x$, and $v$ is odd in
$x$, and \emph{monotone} if in addition $\eta(x)$ is strictly
decreasing for $x > 0$.  
\begin{corollary}[Existence of large-amplitude waves]\label{cor:solitary}
  Fix $g,c,d > 0$, a H\"older parameter $0 < \alpha \le 1/2$, and a
  strictly positive function $U\s \in C^{2+\alpha}[-d,0]$ satisfying
  the normalization condition \eqref{eqn:normalized}. Then there
  exists a connected set $\cm$ of solitary waves 
  \begin{align*}
    (u,v,\eta,F) \in C^{1+\alpha} \by C^{1+\alpha} \by C^{2+\alpha}(\R) \by
    (1,\infty),
  \end{align*}
  where $F$ determines the asymptotic shear flow $U$ via
  \eqref{eqn:family}, with the following properties. Each wave in
  $\cm$ is a symmetric and monotone wave of elevation with $\sup u <
  c$ and $F > 1$. Moreover, at least one of following two conditions
  holds:
  \nat{
  \begin{enumerate}
  \item[\rm (i)] \label{s:part:stag} (Stagnation)
    There is a sequence of flows $(u_n,v_n,\eta_n,F_n) \in \cm$ and
    sequence of points $(x_n,y_n)$ such that $u_n(x_n,y_n) \nearrow c$; or
  \item[\rm (ii)] \label{s:part:height} (Large Froude number)
    There exists a sequence of flows $(u_n,v_n,\eta_n,F_n) \in \cm$
    with $F_n \nearrow \infty$.
  \end{enumerate}
  }{
  \begin{enumerate}[label=\textup{(\roman*)}]
  \item \label{s:part:stag} (Stagnation)
    There is a sequence of flows $(u_n,v_n,\eta_n,F_n) \in \cm$ and
    sequence of points $(x_n,y_n)$ such that $u_n(x_n,y_n) \nearrow c$; or
  \item \label{s:part:height} (Large Froude number)
    There exists a sequence of flows $(u_n,v_n,\eta_n,F_n) \in \cm$
    with $F_n \nearrow \infty$.
  \end{enumerate}
  }
  \begin{proof}
    By Theorem~1.1 in \cite{paper}, it is enough to show that no
    solitary wave $(u,v,\eta,F)$ in the closure of $\cm$
    can have the critical Froude number $F=1$. Let $(u,v,\eta,F)$ be a
    wave in the closure of $\cm$. By Proposition~2.4 in \cite{paper},
    $\sup u < c$, so Theorem~\ref{thm:lower} implies that $F \ne 1$.
  \end{proof}
\end{corollary}
We note that, even when \nat{(i)}{\ref{s:part:stag}} occurs in
Corollary~\ref{cor:solitary}, the shear flow $U$ is bounded away from
$c$ uniformly along the continuum $\cm$. Indeed, every wave in $\cm$
has $F > 1$ and hence
\begin{align*}
  \min (c-U) = F \min U\s > \min U\s > 0.
\end{align*}

In many cases, we can apply Theorem~\ref{thm:upper} to further
simplify Corollary~\ref{cor:solitary}.
\begin{corollary}\label{cor:solitary:extra}
  In the setting of Corollary~\ref{cor:solitary}, suppose that the
  fixed profile $U\s$ satisfies
  \begin{align*}
    \Lambda\s
    := \max_y \frac{U\s(0)}{U\s(y)} < \frac 2{\sqrt 3}.
  \end{align*}
  Then condition \nat{{\rm(ii)}}{\ref{s:part:height}} cannot occur, so that
  \nat{{\rm(i)}}{\ref{s:part:stag}} must hold.
\end{corollary}
\begin{proof}
  Thanks to \eqref{eqn:family}, any wave in $\cm$ satisfies
  \begin{align*}
    \Lambda = \max_y \frac{c-U(0)}{c-U(y)}
    =
    \max_y \frac{U\s(0)}{U\s(y)}
    = \Lambda\s.
  \end{align*}
  Thus by Theorem~\ref{thm:upper} all waves in $\cm$ have $F < (1 -
  \frac 34(\Lambda\s)^2)^{-1/2} < \infty$.
\end{proof}
The conclusion of Corollary~\ref{cor:solitary:extra}, that there
exists a sequence of solutions along the continuum with $\sup u_n$
approaching $c$, is the same conclusion that was proved for periodic
waves \nat{in \citet{cs:exact}}{by Constantin and Strauss in
\cite{cs:exact}}. It remains an open question if the same is true for
solitary waves with $\Lambda\s \ge 2/\sqrt 3$, though it seems
doubtful that the restriction $\Lambda\s < 2/\sqrt 3$ is sharp.

\section{The case of constant vorticity} \label{sec:const} 

In this section we specialize the above results to asymptotic shear
flows $U$ which are linear in $y$, or equivalently to waves whose
vorticity
\begin{align*}
  \omega(x,y) = \gamma(p) \equiv -U_y 
\end{align*}
is constant. In this case more explicit formulas are available;
to make these formulas appear simpler, we define the
dimensionless constants
\begin{align*}
  \lambda\s = \frac{\lambda}{gd} = \frac{(c-U(0))^2}{gd} > 0,
  \qquad 
  \gamma\s = \frac{\gamma d}{\sqrt\lambda} = \frac{-U_yd}{(c-U(0))}.
\end{align*}
In terms of $\lambda\s$ and $\gamma\s$, the asymptotic shear flow $U$
is given by
\begin{align}
  \label{eqn:uconst}
  c-U(y) = \sqrt{\lambda}+\gamma y
  = \sqrt{gd\lambda\s}\left(1 + \gamma\s \frac yd\right).
\end{align}
Plugging $y=-d$ into \eqref{eqn:uconst}, we see that $\sup U < c$
implies $\gamma\s < 1$. Substituting 
\eqref{eqn:uconst} into the definitions \eqref{eqn:froude} of $F$ and
\eqref{eqn:Lambda} of $\Lambda$, we easily check that 
\begin{align}
  \label{eqn:flconst}
  F^2
  = \lambda\s (1-\gamma\s),
  \qquad 
  \Lambda 
  = \frac 1{1-\max(\gamma\s,0)}.
\end{align}
Using \eqref{eqn:flconst} in Theorems~\ref{thm:lower} and
\ref{thm:upper}, we can then easily prove the following.
\begin{corollary}[Constant vorticity]\label{cor:constant}
  Consider a solitary wave of elevation with constant vorticity
  $\gamma$, regularity \eqref{eqn:reg}, and $\sup u < c$ so that in
  particular $\gamma\s < 1$. Then we have the
  following bounds on $\lambda\s$ and $\gamma\s$, 
  \begin{alignat}{2}
    \label{eqn:const:lower}
    \lambda\s (1-\gamma\s) &> 1,\\
    \label{eqn:const:upperneg}
    \lambda\s (1-\gamma\s) &< 4 &\quad&\fora \gamma\s \le 0,\\
    \label{eqn:const:upperpos}
    \lambda\s \frac{1-8\gamma\s+4(\gamma\s)^2}{1-\gamma\s} &< 4 &&\fora 0<\gamma\s< 1-\sqrt{3}/2,
  \end{alignat}
  and hence the bounds
  \begin{align}
    \label{eqn:const:amp}
    \frac{\max \eta}d < \frac 2{1-\gamma\s} \quad\fora \gamma\s \le 0,
    \qquad 
    \frac{\max \eta}d < \frac{2(1-\gamma\s)}{1-8\gamma\s+4(\gamma\s)^2}
    \quad\fora 0<\gamma\s< 1-\sqrt{3}/2.
  \end{align}
  on the amplitude.
  \begin{proof}
    The first inequality \eqref{eqn:const:lower} is $F>1$ from
    Theorem~\ref{thm:lower}, while \eqref{eqn:const:upperneg} and
    \eqref{eqn:const:upperpos} are \eqref{eqn:upper} from
    Theorem~\ref{thm:upper}. Since Bernoulli's law
    \eqref{eqn:Bernoulli} implies $\max \eta < d\lambda\s/2$, the
    bounds \eqref{eqn:const:amp} on the amplitude follow immediately
    from \eqref{eqn:const:upperneg} and \eqref{eqn:const:upperpos}.
  \end{proof}
\end{corollary}
From \eqref{eqn:const:upperneg} and \eqref{eqn:const:upperpos} we see
that, for any fixed $\gamma\s < 1-\sqrt 3/2 \approx 0.134$, waves of
elevation with $\sup u < c$ have $\lambda\s$ bounded above by a
constant depending only on $\gamma\s$. It is interesting to compare
this to the numerical results in \cite{vandenb:constant}, which
suggest that for any fixed $\gamma\s > \gamma\s_{\textup{cr}} \approx
0.33$ there exist overhanging waves with $\lambda\s$ arbitrarily
large. Note however that overhanging waves necessarily violate our assumption
$\sup u < c$.


\section{Surface tension} \label{sec:tension} 

In this section we prove that Theorem~\ref{thm:lower} continues to
hold in the presence of surface tension. For waves with surface
tension, the dynamic boundary condition \eqref{eqn:ww:dynamic} is
replaced by
\begin{align}
  \label{eqn:ww:dynamic:ten}
  P + \sigma \bigg( \frac{\eta_x}{\sqrt{1+\eta_x^2}} \bigg)_x = 0
  \quad \ona y = \eta(x),
\end{align}
where the constant $\sigma$ is the coefficient of surface tension. In
the following we permit $\sigma$ to be positive or negative. The
corresponding boundary condition \eqref{eqn:height:2} in the height
equation becomes
\begin{align}
  \label{eqn:height:2:ten}
  \frac{1+h_q^2}{2h_p^2} - \frac 1{2H_p^2} + g(h-H) 
  - \sigma \bigg( \frac{h_q}{(1+h_q^2)^{1/2}} \bigg)_q &= 0
  \quad \ona p=0.
\end{align}
For irrotational solitary waves with surface tension,
\nat{\citet{ak:tension} prove}{Amick and
Kirchgässner proved} the analogue
\begin{align}
  \label{eqn:toogood:ten}
  F^2 = 1 + \frac 3{2d} \frac{\int \eta^2 \, \dd x}{\int \eta\, \dd x}
  + \frac\sigma{gd} 
  \frac{\int (\sqrt{1+\eta_x^2}-1)\, \dd x}{\int \eta\, \dd x}
\end{align}
of the integral identity \eqref{eqn:toogood}\nat{}{ in
\cite{ak:tension}}. For irrotational waves of depression, such as
those constructed in \cite{ak:tension}, \eqref{eqn:toogood:ten}
immediately implies that the Froude number $F < 1$. As a consequence
of Theorem~\ref{thm:lower:ten} below, the same bound $F< 1$ holds for
waves with vorticity. 

We will work with solutions which have the regularity
\begin{align}
  \label{eqn:reg:ten}
  \eta = h(\placeholder,0) \in W^{2,r}_\loc(\R),
  \qquad 
  h \in W^{2,r}_\loc(\Obar) \sub C^{1+\alpha}_\loc(\Obar),
  \qquad 
  H \in W^{2,r}[-m,0],
\end{align}
where, as in Section~\ref{sec:intro:results}, $0 < \alpha < 1$ and $r
= 2/(1-\alpha)$; see \cite{mm:unbounded} for the equivalence of
various formulations for periodic waves with surface tension.

\begin{theorem}\label{thm:lower:ten}
  Theorem~\ref{thm:lower} holds for waves with
  surface tension, provided we replace the regularity \eqref{eqn:reg}
  with \eqref{eqn:reg:ten}.
  \begin{proof}
    We argue exactly as in the proof of Theorem~\ref{thm:lower}, with
    Lemma~\ref{lem:lower} replaced by Lemma~\ref{lem:lower:ten} below.
  \end{proof}
\end{theorem}
\begin{lemma} \label{lem:lower:ten}
  Lemma~\ref{lem:lower} holds for waves with surface tension, provided
  we replace the regularity \eqref{eqn:reg} with \eqref{eqn:reg:ten}.
  \begin{proof}
    We will follow the proof of Lemma~\ref{lem:lower} and notice that
    the term involving surface tension drops out of the calculation
    entirely. Multiplying \eqref{eqn:height:1} by $\Phi(p) =
    \int_{-m}^p H_p^3(s)\,\dd s$ and integrating by parts, we obtain
    \eqref{eqn:dropme} as before,
    \begin{align}
      \label{eqn:dropme:ten}
      0 &= 
      \int_{-M}^M \int_{-m}^0 
      \Big( \frac{1+h_q^2}{2h_p^2} - \frac 1{2H_p^2} \Big) \Phi_p
      \, \dd p\, \dd q
      +
      \frac 1{gF^2}\int_{-M}^M
      \Big( - \frac{1+h_q^2}{2h_p^2} + \frac 1{2H_p^2} \Big)(q,0) \,
      \dd q\\
      \notag
      &\qquad +
      \int_{-m}^0 \frac{h_q}{h_p} \Phi\, \dd p \bigg|^{x=M}_{x=-M}.
    \end{align}
    We claim that \eqref{eqn:dropme:ten} implies \eqref{eqn:premagic}.
    Indeed, the first and last terms in \eqref{eqn:dropme:ten} can be
    treated as in the proof of Lemma~\ref{lem:lower}, so \eqref{eqn:premagic} follows from
    the following computation involving the middle term,
    \begin{align*}
      \frac 1{gF^2}\int_{-M}^M
      &\Big( - \frac{1+h_q^2}{2h_p^2} + \frac 1{2H_p^2} \Big)(q,0) \,
      \dd q
      -\frac 1{F^2} \int_{-M}^M (h(q,0)-H(0))\, \dd q\\
      &\qquad = - \frac{\sigma}{gF^2} \int_{-M}^M 
        \bigg( \frac{h_q}{(1+h_q^2)^{1/2}} \bigg)_q(q,0)\, \dd q\\
        &\qquad = -\frac \sigma{gF^2} \frac{h_q}{(1+h_q^2)^{1/2}}\bigg|^{(M,0)}_{(-M,0)}
      \to 0 \asa M \to \infty,
    \end{align*}
    where we have used the boundary condition \eqref{eqn:height:2:ten}
    and the asymptotic conditions \eqref{eqn:height:asym}. With
    \eqref{eqn:premagic} established, we can then complete the
    argument exactly as in the proof of Lemma~\ref{lem:lower}.
  \end{proof}
\end{lemma}

\def\cprime{$'$}

\end{document}